\RequirePackage[figuresright]{rotating}
\RequirePackage[skip=1ex]{caption}
\documentclass[trsc,nonblindrev]{informs3} 

\OneAndAHalfSpacedXI 

\usepackage{natbib}
\usepackage{threeparttable}
\usepackage{algorithmic}
\usepackage{lipsum}
\usepackage{amsfonts}
\usepackage{graphicx}
\usepackage{epstopdf}
\usepackage{algorithm}
 \bibpunct[, ]{(}{)}{,}{a}{}{,}%
 \def\bibfont{\small}%
\TheoremsNumberedThrough     

\EquationsNumberedThrough    

\begin{document}


\RUNAUTHOR{Xiaoyu Luo, Mingming XU, and Chuanhou Gao}

\RUNTITLE{ADPBA }

\TITLE{ADPBA: Efficiently generating  Lagrangian cuts for two-stage stochastic integer programs }

\ARTICLEAUTHORS{%
\AUTHOR{Xiaoyu Luo}
\AFF{School of Mathematical Sciences, Zhejiang University, Hangzhou,  China, \EMAIL{12135040@zju.edu.cn}, \URL{}}
\AUTHOR{Mingming Xu}
\AFF{School of Mathematical Sciences, Zhejiang University, Hangzhou,  China, \EMAIL{22135092@zju.edu.cn}, \URL{}}
\AUTHOR{Chuanhou Gao}
\AFF{School of Mathematical Sciences, Zhejiang University, Hangzhou,  China, \EMAIL{gaochou@zju.edu.cn}, \URL{}}
} 

\ABSTRACT{%
The use of Lagrangian cuts proves effective in enhancing the lower bound of the master problem within the execution of benders-type algorithms, particularly in the context of two-stage stochastic programs. However, even the process of generating a single Lagrangian cut is notably time-intensive. In light of this challenge, we present a novel framework that integrates Lagrangian cut generation with an adaptive partition-based approach, thereby mitigating this time-related drawback to a considerable extent.  Furthermore, we also discuss the dominance relationship between the generated partition-based Lagrangian cut and the Lagrangian cut for the original problem. To provide empirical evidence of our approach's efficacy, we undertake an extensive computational study encompassing instances involving even up to a thousand scenarios. The results of this study conclusively demonstrate the superiority and efficiency of the proposed methodology. 
}%


\KEYWORDS{Two-stage stochastic integer program, Partition-based, Lagrangian cut, Fixed continuous recourse}

\maketitle

%


\section{Introduction}
Two-stage stochastic programs are being broadly applied to model many realistic production scenarios, such as the facility location problem (\cite{oliveira2022benders,duran2023efficient}), the railway timetable problem (\cite{leutwiler2022logic}), the network flow problem (\cite{rahmaniani2018accelerating}), etc. Logically, there include deterministic decision variables (often part of them are real while part of them are integral) in the first stage and real deterministic decision variables but affected by random occurrence of scenarios in the second stage. The first-stage decision variables need to be determined before the random variable is revealed, and the second-stage decision variables are determined by an optimization problem parameterized by the first-stage variables and the random variables. The objective function of this class of programs is an expectation of the cost for every realization of the random variable(\cite{birge2011introduction}). \

Benefiting from the technique of sample average approximation that makes an approximation to the distribution of the random variable \cite{kleywegt2002sample,  ahmed2002sample,  kim2015guide}, the following extensive deterministic formulation often acts as a starting point to study a two-stage stochastic integral program (SIP)

\begin{subequations}\label{equa:1}
\begin{align}
	     &\min_{x,y^s} c^{\top} x+\sum_{s \in S} p^sd^{\top} y^s, \label{equa1} \\ 
		\text {s.t.}~~& A x=b,  \\
		&  T^s x+W y^s \geq h^s, ~\forall s \in S, \\
		& x \in \mathbb{R}_{+}^{n_1-p_1} \times \mathbb{Z}_{+}^{p_1},~ 
        y^s \in \mathbb{R}_{+}^{n_2}, ~\forall s \in S, \label{pro1}
	\end{align}
 \end{subequations}
 
where $x$ is the first-stage decision variable, $p^s$ is the probability that scenario $s$ occurs, $c\in\mathbb{R}^{n_1}$, $A \in \mathbb{R}^{m_1\times n_1}$, $b \in \mathbb{R}^{m_1}$, $T^s\in\mathbb{R}^{m_2\times n_1}$ and $h^s\in\mathbb{R}^{m_2}$ are scenario-specific, $W\in \mathbb{R}^{m_2\times n_2}$ is the recourse matrix independent of scenario (fixed recourse), $S = \left\{1, 2, ..., \lvert S \rvert\right\}$ is the set of scenarios, and $y^s$ is the second-stage decision variable determined after the stochastic scenario $s$ is revealed. Note that here the second-stage cost vector $d\in\mathbb{R}^{n_2}$ is set to be not scenario-specific, the main reason of which is for the convenience of aggregating $y_s$ under the the adaptive framework. The optimization problem of (\ref{equa:1}) is essentially a large-scale mixed integer program (MIP), which is referred to as the original problem in the context. \

Generally, it is difficult to directly solve (\ref{equa:1}), especially when the number of scenarios is large. An alternative method is to implement Benders decomposition (\cite{bnnobrs1962partitioning,rahmaniani2017benders}) on the problem to make it tractable. Mathematically, the Benders decomposition approach breaks down (\ref{equa:1}) into a master problem 
\begin{subequations}
\label{equa:mas}
\begin{align}
	   &\min_{x\in X} c^{\top}x + \sum_{s\in S} p^s\theta^{s}, \label{1.3}\\
         \text{s.t.}~~& \theta^{s} \geq F^{s}x + g^{s},~ (F^{s}, g^{s}) \in \mathcal{F}^{s}, ~\forall s\in S\label{equa4}
\end{align}
\end{subequations}
with $\mathbb{X} = \left\{ x: Ax = b, x \in \mathbb{R}_{+}^{n_1-p_1} \times \mathbb{Z}_{+}^{p_1}\right\}$, and a subproblem
 \begin{equation}
	    f^{s}(x) := \min_{y^s\in \mathbb{R}^{n_2}_{+}}\left\{d^{\top} y^s \vert T^s x+W y^s \geq h^s \right\} .\label{equa2}
\end{equation}
In (\ref{equa:mas}), $\mathcal{F}^{s}$ is a collection of the Benders cuts, which are generated by the Benders subproblem (\ref{equa2}) and used to improve the lower bound the Benders master problem (\ref{equa:mas}), and $f^{s}(x)$ in (\ref{equa2}) represents the second-stage value function. Moreover, the Benders cut can be written as this form:
\begin{equation}
    \theta^s \geq \pi^{\top}(h^s - T^sx)
\end{equation}
where $\pi$ is any extreme point of the set $\left\{\pi \in \mathbb{R}^m_{+}: \pi^{\top}W \leq d^{\top}\right\}$.\

Despite being an available method, the strategy of Benders decomposition is not always valid to achieve the optimal solution. The main reason is that Benders decomposition essentially maps the linear relaxation of the original problem onto a lower-dimensional space, which often suffers that the relaxed optimal solution oscillates during the cutting plane process and the generated Benders cuts can not improve it accordingly. As a result, the strategy often causes slow convergence to the optimal solution.\

To tackle the issue of slow convergence, lots of efforts have been made towards accelerating Benders decomposition within the framework of two-stage SIPs with finite scenarios. The main focus is on how to strengthen the Benders cuts. Bodur et al. (\cite{bodur2017strengthened}) presented a unified framework called `cut and project' for two-stage SIP with continuous recourse problem. They proved that cutting the Benders subproblem first can result in tighter relaxation of the Benders master problem than implementing the projection first. Zhang and Kucukyavus (\cite{zhang2014finitely}) cut the subproblem iteratively to approximate the convex hull of the upper graph of the value function $f_{s}(x)$, whose work can handle situations where the second stage also contains integer variables. Rahmaniani et al. (\cite{rahmaniani2020benders}) proposed an innovative technique named Benders dual decomposition, aiming to improve the lower bound of the Benders master problem. This method exhibits a large potential to effectively improve the lower bound, but at the expense of transforming its associated subproblem from a linear program to a MIP. The resultant cut generated by this subproblem is termed `Lagrangian cut' (\cite{chen2022generating}) due to its close relationship with the Lagrangian dual decomposition (\cite{caroe1999dual}). There are also some studies on creative algorithms development based on the problem structure for accelerating Benders decomposition. Song and Luestke (\cite{song2015adaptive}) proposed a framework called `adaptive scenario partition' for two-stage stochastic linear programming (SLP). They demonstrated the existence of a small sufficient partition in the case of a simple recourse problem. Under the assumption of fixed recourse, the fundamental concept behind the scenario partition is to cluster the scenario set and construct a relatively coarse lower approximation for the second-stage value function by aggregating the scenarios within the same cluster. Pay and Song (\cite{pay2020partition}) integrated the aforementioned framework with `branch-and-cut' techniques to reduce the time consumed to solve the second-stage recourse problem. However, they did not make any attempt to address integer constraints at the root node. Further, Ackooij et al. (\cite{van2018adaptive}) proposed an adaptive partition level decomposition while Ram\'irez-Pico, et al. (\cite{ramirez2023benders}) proposed adaptive Bender decomposition for two-stage SLP with fixed recourse. \

In line with these studies, this paper is also concerned about the issue of how to accelerate Benders decomposition. Note that the Lagrangian cut is effective in improving the lower bound of the Benders master problem (\cite{rahmaniani2020benders}) and the adaptive partition-based strategy can reduce the time to solve subproblem (\cite{song2015adaptive}), we thus try to develop the algorithm of adaptive partition-based Benders dual decomposition to generate Lagrangian cuts (APbLagC) by aggregating scenarios for two-stage SIP of (\ref{equa:1}), and further solve it in a short time. The main contributions of this paper can be summarized as follows:
\begin{itemize}
    \item[(i)] develop APbLagC algorithm to efficiently generate Lagrangian cuts for addressing two-stage SIPs with fixed recourse and continuous second-stage variables;
    \item[(ii)] prove that in the case of $\dim x=1$ there is no Lagrangian duality gap for the concerned two-stage SIP, which further leads to the monotonicity of the optimal values of the Lagrangian dual relaxations concerning the partition-based problem along the refinement order of scenario partitions;
    \item[(iii)] prove that the Lagrangian relaxations of the partition-based problem might be tighter than those of the original problem (\ref{equa:1}), which stands in sharp contrast to the situation of their linear relaxations;
    \item[(iv)] conduct numerical experiments to show that APbLagC algorithm is effective in lifting the lower bound of the master problem through adding partition-based Lagrangian cut.
\end{itemize}

The rest of this paper is organized as follows. Section 2 presents preliminaries about the scenario partition formulation of the two-stage SIP and the Benders dual decomposition. This is followed by Section 3 where APbLagC algorithm is proposed and the generated partition-based Lagrangian cut is proved to be a valid inequality for Benders formulation. In Section 4, some theoretical analyses are made on the dominance of the partition-based Lagrangian cut. Section 5 exhibits some computational experiments to support the proposed APbLagC algorithm and the theoretical results. Finally, we concludes our work in Section 6. 

\section{Preliminaries}
\label{sec:Background}
In this section, we provide a brief introduction to the scenario partition formulation of two-stage SIP (\cite{song2015adaptive}) and Benders dual decomposition (\cite{chen2022generating}).

\subsection{The partition-based formulation of two-stage SIP}
For the two-stage SIP of (\ref{equa:1}), the main idea of writing its partition-based formulation is to aggregate scenarios according to a certain partition of scenarios. Assume that $\mathcal{N} = \left\{\mathcal{P}_1,\mathcal{P}_2,...,\mathcal{P}_L\right\}$ is a kind of partition of the scenario set $S$ satisfying $\mathcal{P}_1\cup \mathcal{P}_2\cup ...\cup \mathcal{P}_L = S$ and $\mathcal{P}_i \cap \mathcal{P}_j = \emptyset$, $\forall i$, $j \in \{1,2,...,L\}$ and $i \neq j$. Then the partition-based problem can be stated as follows:
\begin{subequations}\label{equa:parti}
\begin{align}
	 &\min_{x,y^\mathcal{P}} c^{\top} x+\sum_{\mathcal{P} \in \mathcal{N}} p^sd^{\top} y^{\mathcal{P}}, \label{2.1}\\
		\text { s.t.}~~&\bar{T}^{\mathcal{P}} x+W y^{\mathcal{P}} \geq \bar{h}^{\mathcal{P}}, ~\forall \mathcal{P} \in \mathcal{N},\label{2.2} \\
		& x \in\mathbb{X},~ y^{\mathcal{P}} \in \mathbb{R}_{+}^{n_2}, ~\forall \mathcal{P} \in \mathcal{N},\label{equa6} 
	\end{align}
 \end{subequations}
 where $y^{\mathcal{P}}=\frac{\sum_{s\in\mathcal{P}}p^s y^s}{\sum_{s\in\mathcal{P}}p^s}$, $\bar{T}^{\mathcal{P}}=\frac{\sum_{s \in \mathcal{P}} p^s T^s}{\sum_{s\in\mathcal{P}}p^s}$, and $\bar{h}^{\mathcal{P}}=\frac{\sum_{s \in \mathcal{P}}p^s h^s}{\sum_{s\in\mathcal{P}}p^s}$. Compared with the original problem of (\ref{equa:1}), the current one seems to involve fewer scenarios and is accordingly more easily to solve. We view it as a relaxation of (\ref{equa:1}), and name it the aggregated problem. Naturally, different partitions will cause different partition-based formulations. The following concept of refinement is used to distinguish them. 

 \begin{definition} [Refinement
  (\cite{song2015adaptive})] Let $\mathcal{N}_1$, $\mathcal{N}_2$ be two kinds of scenario partitions of S, then $\mathcal{N}_1$ is a refinement of $\mathcal{N}_2$ if $\forall \mathcal{P} \in \mathcal{N}_1$, $\mathcal{P}\subset \mathcal{P}^{'}$ for some $\mathcal{P}^{'} \in \mathcal{N}_2$ and $\lvert \mathcal{N}_1 \rvert > \lvert \mathcal{N}_2 \rvert$.
  \end{definition}

  We give an example to exhibit refinement.
\begin{example}
    Let $S = \left\{1, 2, 3, 4\right\}$ be the scenarios set, and $\mathcal{N}_1 = \left\{\left\{1\right\}, \left\{2\right\}, \left\{3, 4\right\}\right\}$ and $\mathcal{N}_2 = \left\{\left\{1, 2\right\}, \left\{3,4\right\}\right\}$ are two kinds of scenario partitions, then $\mathcal{N}_1$ is a refinement of $\mathcal{N}_2$.
\end{example}

The refinement notion will play an important role on approximating (\ref{equa:1}) through (\ref{equa:parti}). \cite{song2015adaptive} developed the adaptive partition-based algorithm, given in  Algorithm \ref{algo1}, to follow this logic. The algorithm iteratively solves the aggregated problem by adaptively updating the partition $\mathcal{P}$ (towards more refined partition) according to the second-stage optimal dual solutions. As a result of refinement every time, the aggregated problem (\ref{equa:parti}) will be closer to the original problem (\ref{equa:1}). Note that in this algorithm it is designed to solve the aggregated problem (\ref{equa:parti}) directly without considering the partition information in the previous steps. \cite{van2018adaptive,pay2020partition} further developed adaptive partition based level decomposition algorithms that utilize these information. We will follow this line to push our work.  

\begin{algorithm}
 \caption{Adaptive Partition-based Algorithm \cite{song2015adaptive}}
 \label{algo1}
 \begin{algorithmic}
     \STATE $z^{ub} \leftarrow +\infty$, assign an initial optimal value $z^{ub}$ of (\ref{equa:parti}) and choose an initial scenario partition $\mathcal{N}$. 
     \WHILE{gap $> \epsilon$}
     \STATE Solve (\ref{equa:parti} )to obtain $\hat{x}$ and the optimal value $z^{\mathcal{N}}$.
    \FOR{$s \in S$} 
     \STATE Solve Benders subproblem (\ref{equa2}) with $x = \hat{x}$ and obtain an optimal dual solution $\hat{\lambda}^{s}$.
    \ENDFOR
    \STATE Refine partition $\mathcal{P}$ based on these optimal dual solution.
    \STATE $z^{ub} \leftarrow \min\left\{z^{ub},c^{\top}\hat{x} + \sum p^sf_{s}\right\}$
    \STATE gap $\leftarrow \frac{z^{ub} - z^{\mathcal{N}}}{z^{ub}}$
    \ENDWHILE
    \end{algorithmic}
    \end{algorithm}

\subsection{Benders dual decomposition}
Benders dual decomposition (\cite{rahmaniani2020benders}) provides a kind of effective way to improve the lower bound of the Benders master problem, but the separation problem to generate Lagrangian cuts is not designed well. The situation was improved by Chen and Luedtke (\cite{chen2022generating}) who reformulated the original problem of (\ref{equa:1}) as the following Benders model:
\begin{subequations}\label{equa8}
\begin{align}
	   & \min _{x, \theta^s}\left\{c^{\top} x+\sum_{s \in S} p_s \theta_s:\left(x, \theta^s\right) \in E^s,~ s \in S\right\},\label{equa8a}\\
  & E^s=\left\{\left(x, \theta_s\right) \in \mathbb{X} \times \mathbb{R}: A x \geq b,~ \theta_s \geq Q_s(x)\right\}, \label{equa8b}\\
  & Q_s(x)=\min _y\left\{d^{\top} y: W y \geq h^s-T^s x, ~y \in \mathbb{R}^{n_2}_{+}\right\}.\label{equa8c}
	\end{align}
 \end{subequations}
 Note that there include integer constraints on the first-stage variables $x$, so unlike the generation of Benders cuts where a linear programming Benders subproblem is solved, it needs to solve a MIP to generate Lagrangian cuts. For this purpose, they designed the optimization problem
\begin{subequations}
\begin{align}
	\bar{Q}_s^*\left(\pi, \pi_0\right)=&\min _x\left\{\pi^{\top} x+\pi_0 Q_s(x): A x \geq b, x \in \mathbb{X}\right\}\\
    =&\min _{x, y}\left\{\pi^{\top} x+\pi_0\left(d\right)^{\top} y:(x, y) \in K^s\right\},
 \end{align}
\end{subequations}
where $(\pi,\pi_0) \in (\mathbb{R}^{n}\times \mathbb{R}^{+})$ and $K^s:=\left\{x \in \mathbb{X}, y \in \mathbb{R}^{n_2}_{+}: A x \geq b, T^s x+W^s y \geq h^s\right\}$. The separation problem thus follows
  \begin{align}
    L_s(\hat{x})=\max_{(\pi,\pi_{0})\in \pi_{s}}\left\{\Bar{Q}_{s}^{\ast}(\pi,\pi_0) - \pi^{\top}\hat{x} - \pi_0\hat{\theta}^{s}:(\pi,\pi_0) \in \Pi_{s}\right\}\label{equa9},
\end{align}
with $\Pi_{s}$ to be any compact subset of $\mathbb{R}^{n}\times \mathbb{R}^{+}$. The inequality 
\begin{equation}\label{LagC}
 \pi^{\top} x+\pi_0 \theta^s \geq \bar{Q}_s^*\left(\pi, \pi_0\right)   
\end{equation}
is called a `Lagrangian cut', which is also a valid inequality for the Benders master problem (\ref{equa:mas}). The reason that Lagrangian cuts can improve its lower bound may be ascertained by the following lemma.

\begin{lemma}
[Strength of Lagrangian cut ](\cite{chen2022generating})The feasible region defined by all the Lagrangian cuts is equivalent to that of the Lagrangian dual relaxation obtained by dualizing the so-called `nonanticipativity constraint', that is 
\begin{equation}
\label{strength}
\left\{{(x,\theta^{s})_{s\in S}: (x,\theta^{s})\in conv(E^{s})}, ~\forall s\in S\right\}.
\end{equation}
\label{Lag}
\end{lemma}

Clearly, the generated Lagrangian cuts manage to characterise the convex hull of $E^s$ while the Benders cuts could only characterise the linear relaxation of $E^s$ for every scenario $s \in S$. Therefore, the Lagrangian cut is much more effective in improving the lower bound and accelerating the algorithmic convergence.  

\section{The partition-based Lagrangian cut}
Despite proving powerful in improving the lower bound of the Benders master problem, it is very time-consuming to generate Lagrangian cuts due to solving MIP subproblems and the sheer number of scenarios. To increase the efficiency of generating Lagrangian cuts, we borrow the technique of aggregating scenarios adopted to Benders cuts (\cite{pay2020partition}) to ``reduce" the number of scenarios and accordingly to shorten the solving time. We implement this technique on the partition-based problem (\ref{equa:parti}). The generated Lagrangian cut in this way is termed as the ``partition-based Lagrangian cut" (PbLagC) in the context. Imitating (\ref{LagC}), we can directly write out the current one to be
	\begin{align}
		\pi^{\top} x+\pi_0 \theta^\mathcal{P} \geq \bar{Q}_\mathcal{P}^*\left(\pi, \pi_0\right),\label{equa10}
	\end{align}
 
where $\theta^\mathcal{P}=\frac{\sum_{s \in \mathcal{P}} p^s\theta^s}{\sum_{s \in \mathcal{P}} p^s}$, $ \bar{Q}_\mathcal{P}^*\left(\pi, \pi_0\right) = \min _{x, y}\left\{\pi^{\top} x+\pi_0\left(d\right)^{\top} y:(x, y) \in K^\mathcal{P}\right\}$ and similarly, $K^\mathcal{P}:=\left\{x \in \mathbb{X},~ y \in \mathbb{R}^{n_2}_{+}: A x \geq b, ~T^\mathcal{P} x+W y \geq h^\mathcal{P}\right\}$. It is straightforward to get the following proposition. 

\begin{proposition}
\label{P1}
The partition-based Lagrangian cut (\ref{equa10}) is a valid inequality for the Benders formulation (\ref{equa8}).
 \end{proposition}

 \proof{Proof}
    Given a $\mathcal{P} \in \mathcal{N}$, since every feasible point in the Benders formulation (\ref{equa8}) can be stated as $\left\{(x,\theta^s)_{s \in \mathcal{P}}:(x,\theta^s) \in K^s  \ for \ s \in \mathcal{P}\right\}$, $(x,\theta^\mathcal{P}) $ belongs to $K^\mathcal{P}$. Also, from the expression of $ \bar{Q}_\mathcal{P}^*\left(\pi, \pi_0\right)$, we have $\pi^{\top} x+\pi_0 \theta^\mathcal{P} \geq \bar{Q}_\mathcal{P}^*\left(\pi, \pi_0\right)$, which means inequality (\ref{equa10}) is valid for the Benders formulation (\ref{equa8}).
 \endproof

Proposition 1 implies that the inequality (\ref{equa10}) is qualified to enhance the lower bound of the Benders master problem. Some practical instances, like vehicle routing problem with stochastic demands (\cite{florio2020new}), also demonstrate the enhancement efficacy through partitioning scenarios. It is expected that the inequality (\ref{equa10}) can significantly improve the lower bound in high efficiency. For this purpose, PbLagCs are embedded into the Benders-type branch-and-cut framework. Here, we only focus on the root node, and propose APbLagC algorithm to finish the setting task, given in Algorithm \ref{a1}.   
\begin{algorithm}
\caption{APbLagC algorithm}
\label{a1}
\begin{algorithmic}
\STATE $z^{lb} \leftarrow -\infty, z^{ub} \leftarrow +\infty, 
k = 0,\kappa_1 > 0 $, choose an initial partition, e.g., $\mathcal{N} = \left\{\left\{1,2,...,\lvert S\rvert\right\}\right\}$.
\STATE{$refinement \leftarrow False$}.
\WHILE{$continue = True$}
\IF{$refinement = True$}
\STATE $k = k + 1.$
\STATE $LB_{k} \leftarrow$ an empty list.
\STATE update scenario partition by the refinement technology.
\STATE $refinement \leftarrow False$.
\ENDIF
\STATE $cutfound\leftarrow True$.
\WHILE{$cutfound = True$}
\STATE $cutfound \leftarrow False$.

\IF{any Benders cut can be generated}
\STATE $cutfound \leftarrow True$
\ENDIF
\ENDWHILE
\IF{$cutfound = False$}
\FOR{$\mathcal{P}\in \mathcal{N}$}
\STATE generate the  Lagrangian cut associated with this partition.
\IF{A violated Lagrangian cut is generated}
\STATE $cutfound \leftarrow True$.
\ENDIF
\ENDFOR
\STATE Compute the lower bound $z^{lb}$ of the master problem, $LB_{k} = LB_{k} + z^{lb}$.
\IF{the gap closed by the last five rounds is less than $5\%$ of the total gap closed by the current scenario partition}
\STATE $refinement \leftarrow True$.
\ENDIF
\IF{$refinement = True$ and $LB_k[-1] - LB_k[0] < \kappa_1*(LB_k[-1] - LB_0[0])$}
\STATE $continue \leftarrow False$.
\ENDIF
\ENDIF
\ENDWHILE
\end{algorithmic}
\end{algorithm}

 In this algorithm, PbLagCs will work together with PbBenCs to improve the lower bound of the Benders master problem. Firstly, at a given scenario partition $\mathcal{N}$, PbBenCs are added to the branch-and-cut tree to improve lower bound after they are generated from solving the partition-based Benders subproblem 
 \begin{equation}
    f_{\mathcal{P}}(\hat{x}) = \mathop{\min}_{y^{\mathcal{P}}\in \mathcal{R}_{+}^{n_{2}}}\left\{d^{\top} y^{\mathcal{P}}\lvert Wy^{\mathcal{P}} \geq \Bar{h}^{\mathcal{P}} -\Bar{T}^{\mathcal{P}}\hat{x}\right\}\label{equa11}
\end{equation}
with $\mathcal{P}\in \mathcal{N}$ and $\hat{x}$ to be a first-stage solution. This process will go on until there are no more Benders cuts that can be separated. Then, PbLagCs are generated through solving the partition-based Lagrangian subproblem
\begin{equation}
    L_{\mathcal{P}}(\hat{x}) = \max_{(\pi,\pi_{0})\in \pi_{P}}\left\{\Bar{Q}_{\mathcal{P}}^{\ast}(\pi,\pi_0) - \pi^{\top}\hat{x} - \pi_0\hat{\theta}_{\mathcal{P}}:(\pi,\pi_0) \in \Pi_{\mathcal{P}}\right\}\label{equa12}
\end{equation}

for each $\mathcal{P}$, where $\Pi_{\mathcal{P}}$ is also any compact subset of $\mathbb{R}^{n}\times \mathbb{R}^{+}$, and are further added to the branch-and-cut tree to observe the change of lower bound of the Benders master problem. When the lower bound fails to decrease significantly, we update the scenario partition towards more refined one and make a new round of solving until the termination condition is true. The detailed refinement operation (\cite{pay2020partition,song2015adaptive}) follows:
\begin{itemize}
    \item [(1)] at a given partition $\mathcal{N}$, $\forall\mathcal{P}\in \mathcal{N}$, denote the optimal dual multipliers of (\ref{equa2}) by $\hat{\lambda}^{s}$ for any $s\in \mathcal{P}$;
    \item[(2)] let $\left\{\mathcal{K}^{1},...,\mathcal{K}^M\right\}$ be a partition of $\mathcal{P}$ such that $\vert \hat{\lambda}^{s} - \hat{\lambda}^{s'} \vert \leq \delta, \forall s,s' \in \mathcal{K}^{m}$ and $\forall m = 1,...,M$, where $\delta > 0$ is a given threshold;
    \item[(3)] remove $\mathcal{P}$ from $\mathcal{N}$ and add components $\mathcal{K}^{1},...,\mathcal{K}^M$ to it.
\end{itemize}

The solving will stop at a certain scenario partition $\mathcal{N}$. If $\lvert \mathcal{N}\rvert \ll \lvert S \rvert$, the time to generate PbLagCs will significantly decrease. 


\section{Dominance analysis}
 \label{dominance}
In this section, we give some theoretical analysis on dominance relation between Lagrangian cuts and PbLagCs to support Algorithm \ref{a1}. Here, the dominance is defined as follows. 

\begin{definition}\label{def:dominance}
Given two classes of benders-type inequalities
$$ \left\{\theta \geq \alpha_{1}^ix + \beta_1^i
 , i \in I\right\}~ \text{and}~ \left\{\theta \geq \alpha_{2}^jx + \beta_2^j, j \in J\right\}, 
 $$
the latter is said to dominate the former in the weak sense if $\forall x \in \mathbb{X}$ that satisfies the latter must satisfy the former.
 \end{definition}
 
The following proposition exhibits the dominance relation between Benders cuts and PbBenCs.   
 \begin{proposition}
\label{P4.2}
Denote any extreme point of the dual problem of (\ref{equa11}) by $\hat{\lambda}^\mathcal{P}$, then PbBenCs given by 
\begin{equation}\label{eq:PbBenCs}
\theta^\mathcal{P}\geq \left(\Bar{h}^\mathcal{P}-\Bar{T}^\mathcal{P}x\right)^{\top} \hat{\lambda}^\mathcal{P}
\end{equation}
are valid inequalities for the Benders formulation of (\ref{equa8}). Moreover, they must be dominated by some Benders cuts.
 \end{proposition}

\begin{proof}{Proof}
Note that it is not a new result that PbBenCs of (\ref{eq:PbBenCs}) are valid for (\ref{equa8}) (\cite{pay2020partition}). However, we will provide a more concise proof towards it which allows to get dominance relation meanwhile.   

Since (\ref{equa:1}) has a constant recourse matrix $W$, the set of extreme points of the dual problem for the partition-based Benders subproblem (\ref{equa11}) is the same as that for the Benders subproblem (\ref{equa2}). Therefore, the PbBenCs given in (\ref{eq:PbBenCs}) are in fact a convex combination of Benders cuts, so they are also valid for (\ref{equa8}), and moreover, they are weaker than the corresponding Benders cuts according to definition \ref{def:dominance}.
\end{proof}

Proposition \ref{P4.2} implies that the role of every partition-based Benders cut can be replaced by some Benders cuts. Therefore, PbBenCs can not cut off any point in the feasible region of the linear relaxation of the Benders master formulation (\ref{equa:mas}). At this point, the result seems discouraging to demonstrate the role of scenario partition. However, for Lagrangian cuts the situation may be quite different. To this argument, we introduce an extension of Jenson's inequality (\cite{pavic2014convex}).

\begin{lemma}
[An extension of Jenson's inequality]
 (\cite{pavic2014convex}) Consider a convex function $g:\mathbb{R}^n\to\mathbb{R}$, and an integer cubic with volume 1 and vertexes represented by $\left\{x_i\right\}_{i = 1}^{2^n}$. For a point $p$ within the cubic that can be expressed by a convex combination of these vertexes, that is $p = \sum_{i}\lambda_ix_i$, then for any other convex collection of integer points $\left\{y_j\right\}_{j = 1}^{m}$ for $p$, i.e., $p = \sum_{j}\nu_jy_j$, there is
 \begin{align}
     g(p) &\leq \sum_{i}\lambda_ig(x_i)\leq \sum_{j}\nu_jg(y_j).
 \end{align}\label{Convex}
 \end{lemma}

 In addition, under the circumstance of dominance analysis of PbLagCs, based on (\ref{Lag}) the recourse value $f^s$ for scenario $s$ at a given first-stage variable $\hat{x}$ is referred to as the minimum value of the upper graph of the Lagrangian dual relaxation, i.e., $f^s(\hat{x}) = \min \left\{\theta^s, (\hat{x}, \theta^s) \in conv(E^s)\right\}$. Note that this reference is different from (\ref{equa2}) except when $\hat{x}$ is an integer point.\

Utilizing Lemma \ref{Convex} and the above reference of the recourse value, we can analyze the Lagrangian dual gap for the original problem in the case of 1-dimensional first-stage integral variable, and further discuss the corresponding partition-based problem.

\begin{proposition}
In the case of 1-dimensional first-stage integral variable of the original problem, i.e.,    $$
	\begin{aligned}
		& \min c^{\top} x+\sum_{s \in S} p^sd^{\top}y^s, \\
		& \qquad T^s x+Wy^s \geq h^s ,~\forall s \in S, \\
		& \qquad x \in \mathbb{X},~ \mathbb{X} \subseteq \mathbb{Z},~y^s \in \mathbb{R}_{+}^{n_2},~ \forall s \in S,
	\end{aligned}
	$$
we have
\begin{itemize}
    \item [(1)] there is no Lagrangian duality gap for this SIP problem;
    \item[(2)] $\mathcal{N}_1,~ \mathcal{N}_2$ are two kinds of partitions of $S$ such that $\mathcal{N}_1$ is a refinement of $\mathcal{N}_2$, then $z^{\mathcal{N}_1} \geq z^{\mathcal{N}_2} $, where $z^{\mathcal{N}}$ denotes the optimal solution of (\ref{equa:parti}).
\end{itemize}
\end{proposition}

\begin{proof}{Proof}
 We continue the proof in two separate items:
 
(1) the given SIP problem can be reformulated as:
	$$\begin{aligned}
		& \min c^{\top} x+\sum_{s \in S} p_s\theta^s, \\
		& \qquad \theta^se \geq \bar{h}^s - \bar{T}^sx ,~\forall s \in S, \\
		& \qquad x \in \mathbb{X},~ \mathbb{X} \subseteq \mathbb{R},~\theta^s \in \mathbb{R}, ~\forall s \in S,
	\end{aligned}$$
where $e$ is a dimension-suited vector with all entries to be 1. We can rewrite $E^s$ to be $E^s = \left\{(x,\theta^s):x \in \mathbb{X},\theta^s e \geq \bar{h}^s - \bar{T}^sx \right\}$ and try to prove the claim that $\forall s \in S$ and $\forall (x,\theta^s) \in conv(E^s)$, there is $(x,\theta^s) = \lambda_1(x_1,\theta_1^s) + \lambda_2(x_2,\theta_2^s)$, where non-negative weights $\lambda_1+\lambda_2=1$ and $(x_1,\theta_1^s),~(x_2,\theta_2^s) \in E^s$.
	
Clearly, the first-stage integral variable $x$ can be written as $x = \overline{\mu}\lfloor x \rfloor + \underline{\mu}\lceil x \rceil$, where $\overline{\mu} = \lceil x \rceil - x$ and $ \underline{\mu} = x - \lfloor x \rfloor$ satisfying $\overline{\mu}+\underline{\mu}=1$. Since $(x,\theta^s) \in conv(E^s)$, it can be represented as a convex combination of points in $E^s$, that is $(x,\theta^s) = \sum_{i = 1}^{n}\lambda_i (x_i,\theta_i),$ where $(x_i,\theta_i) \in E^s$. Therefore, $\theta^s \geq \sum_{i = 1}^{n}\lambda_i Q_s(x_i) \geq \overline{\mu}Q_s(\lfloor x \rfloor) + \underline{\mu}Q_s(\lceil x \rceil)$, where the last inequality comes from Lemma \ref{Convex}. Hence, by setting $\lambda_1=\overline{\mu}$ and $\lambda_2=\underline{\mu}$ the above claim is true. Namely, any point in the region of Lagrangian relaxation can be represented as a convex combination of two feasible points. This means there is no Lagrangian duality gap for the considered SIP.
 
(2) The result is straightforward from that of (1).
\end{proof}

This proposition suggests that the Lagrangian dual gap vanishing may result from the unifying convex combination expression for different Lagrangian feasible point in different scenario, which conversely explains where the Lagrangian dual gap comes from, i.e., from inconsistent convex combination expression. In the following, we present a sufficient condition to say the recourse value will decrease when scenarios are aggregated, which is the same phenomenon as the linear relaxation.   

\begin{proposition}\label{pro:parrecvalue}
 For the original SIP problem (\ref{equa:1}), assume $\hat{x}$ to be a fixed first-stage solution, $s_1$, $s_2$ to be two scenarios with probability weights $p_{s_1}$ and $p_{s_2}$, respectively, and they are aggregated into class $\mathcal{P}=\{s_1,s_2\}$. If for  points $(x_i, \theta^{s_j}_i) \in E^{s_j},~i \in I,~j=1,2$, there are $(\hat{x}, f^{s_1}(\hat{x})) = \sum_{i \in I}\lambda_i(x_i, \theta^{s_1}_i)$, and $(\hat{x}, f^{s_2}(\hat{x})) = \sum_{i \in I}\lambda_i(x_i, \theta^{s_2}_i)$, where $\lambda_i > 0$ and $\sum_{i \in I}\lambda_i = 1$, then we have 
 \begin{equation}\label{ParRecValue}
 f_{\mathcal{P}}(\hat{x}) \leq \frac{p_{s_1}f^{s_1}(\hat{x}) + p_{s_2}f^{s_2}(\hat{x})}{p_{s_1} + p_{s_2}},
 \end{equation}
 where $f_{\mathcal{P}}(\hat{x})$ follows the expression of (\ref{equa11}).
\end{proposition}

\begin{proof}{Proof}
From $E^{\mathcal{P}}=\left\{\left(x, \theta^{\mathcal{P}}\right) \in \mathbb{X} \times \mathbb{R}: A x \geq b,~ \theta^{\mathcal{P}} \geq Q_{\mathcal{P}}(x)\right\}$, we have $(x_i, \theta^{\mathcal{P}}_{i}) \in E^{\mathcal{P}}$, where $\theta^{\mathcal{P}}_{i} = \frac{p_{s_1}\theta^{s_1}_{i} + p_{s_2}\theta^{s_2}_{i}}{p_{s_1} + p_{s_2}}$. Also, since $(\hat{x}, f^{s_j}(\hat{x})) = \sum_{i \in I}\lambda_i(x_i, \theta^{s_j}_i)$, we get $(\hat{x}, \hat{\theta}^{\mathcal{P}})\in conv(E^{\mathcal{P}})$. Further, from $p_{s_1}(T^{s_1}\hat{x} + W\hat{y}^{s_1}) + p_{s_2}(T^{s_2}\hat{x} + W\hat{y}^{s_2}) \geq p_{s_1}h^{s_1} + p_{s_2}h^{s_2}$, where $\hat{y}^{s_j}$ is the optimal solution for $f^{s_j}(\hat{x})$, we obtain (\ref{ParRecValue}) immediately. 
\end{proof}

\begin{remark}
Proposition \ref{pro:parrecvalue} renders that if the mentioned condition is true, then PbLagCs are dominated by some Lagrangian cuts, which implies that if this condition is not true, it is possible to get dominating PbLagCs compared with the Lagragian cuts.  
\end{remark}

\begin{theorem}\label{StrongPLgC}
PbLagCs are not necessarily dominated by any Lagrangian cut for the original SIP problem of (\ref{equa:1}).
\end{theorem}

\begin{proof}{Proof}
We use an simple example to support the theorem, which includes two scenarios with equal probability of occurrence, and has the feasible domain characterized by the following two linear systems:
	$$
	\begin{array}{ll}
		
        (I):~\left\{\begin{array}{l}
        z \geq x-y \\
        z \geq -x+y
        \end{array}\right.,
       ~~~~~~~
        (II):~\left\{\begin{array}{l}
        z \geq 1 -x - y \\
        z \geq x + y -1
        \end{array}\right..
	\end{array}
	$$
Here, $(x,y) \in \left\{0,1\right\}^2$ are the first-stage variables and $z \in \mathbb{R}$ is the second-stage variable. Figure 1 illustrates the feasible domain. As can be seen, 
in the point $(\frac{1}{2}, \frac{1}{2}, 0)$ the weighted mean of the two graphs in terms of $z$ axis is strictly lower than $\frac{1}{2}$, which results in the Lagrangian duality gap. When the scenario partition technique is adopted, the aggregated problem becomes 
$ \left\{\begin{array}{l}
        z \geq \frac{1}{2} - y \\
        z \geq y - \frac{1}{2}
        \end{array}\right.,$ which even eliminates the duality gap of the Lagrangian relaxation. Based on Lemma \ref{Lag}, the feasible region defined by all the Lagrangian cuts is equivalent to that of the Lagrangian dual relaxation, so we get the statement of Theorem \ref{StrongPLgC}.

\begin{figure}[H]\label{CaseFig}
		\centering
		\begin{minipage}[t]{0.48\textwidth}
			\centering
			\includegraphics[width=4cm]{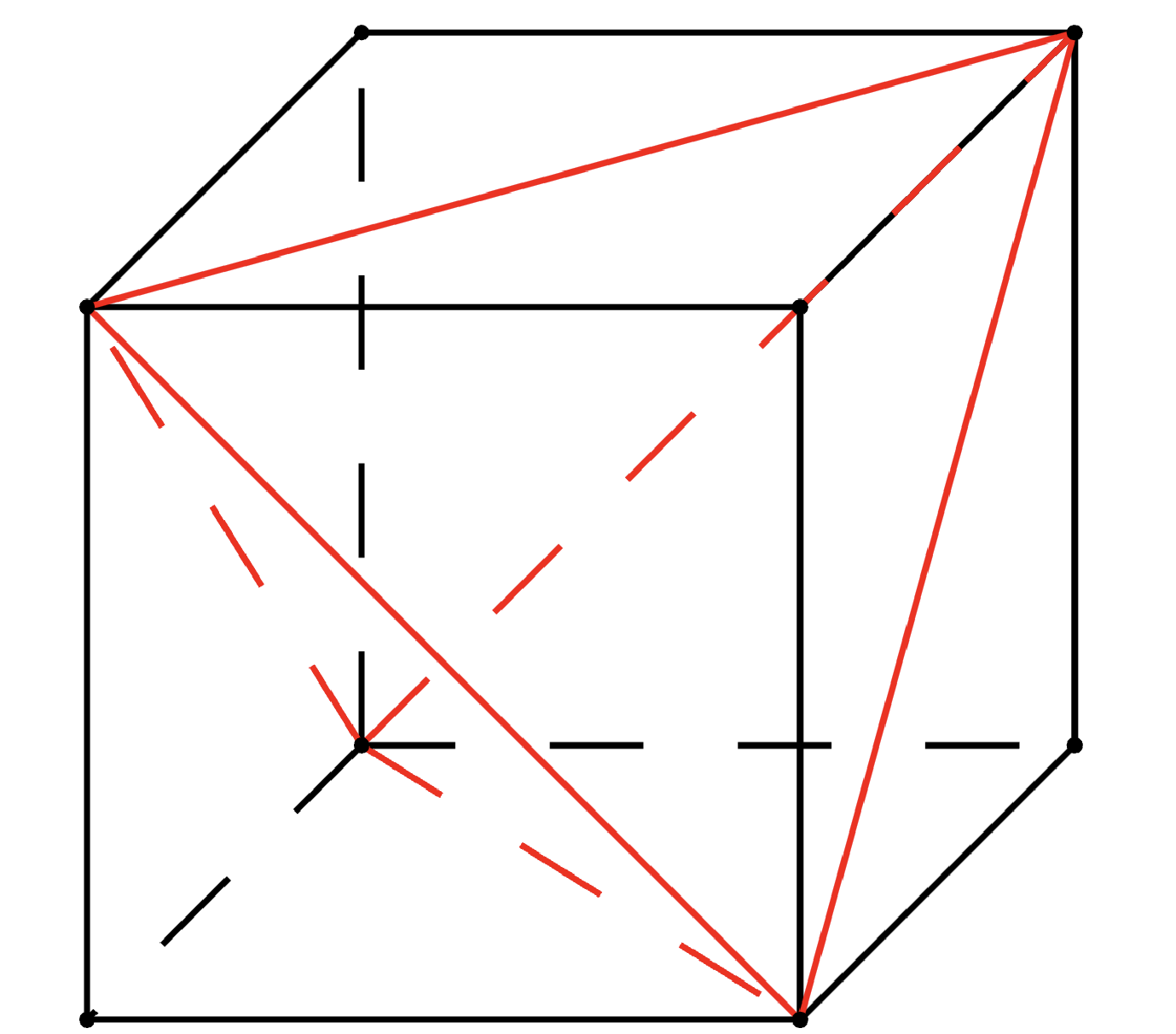}
			\centerline{(a)}
		\end{minipage}
		\begin{minipage}[t]{0.48\textwidth}
			\centering
			\includegraphics[width=4cm]{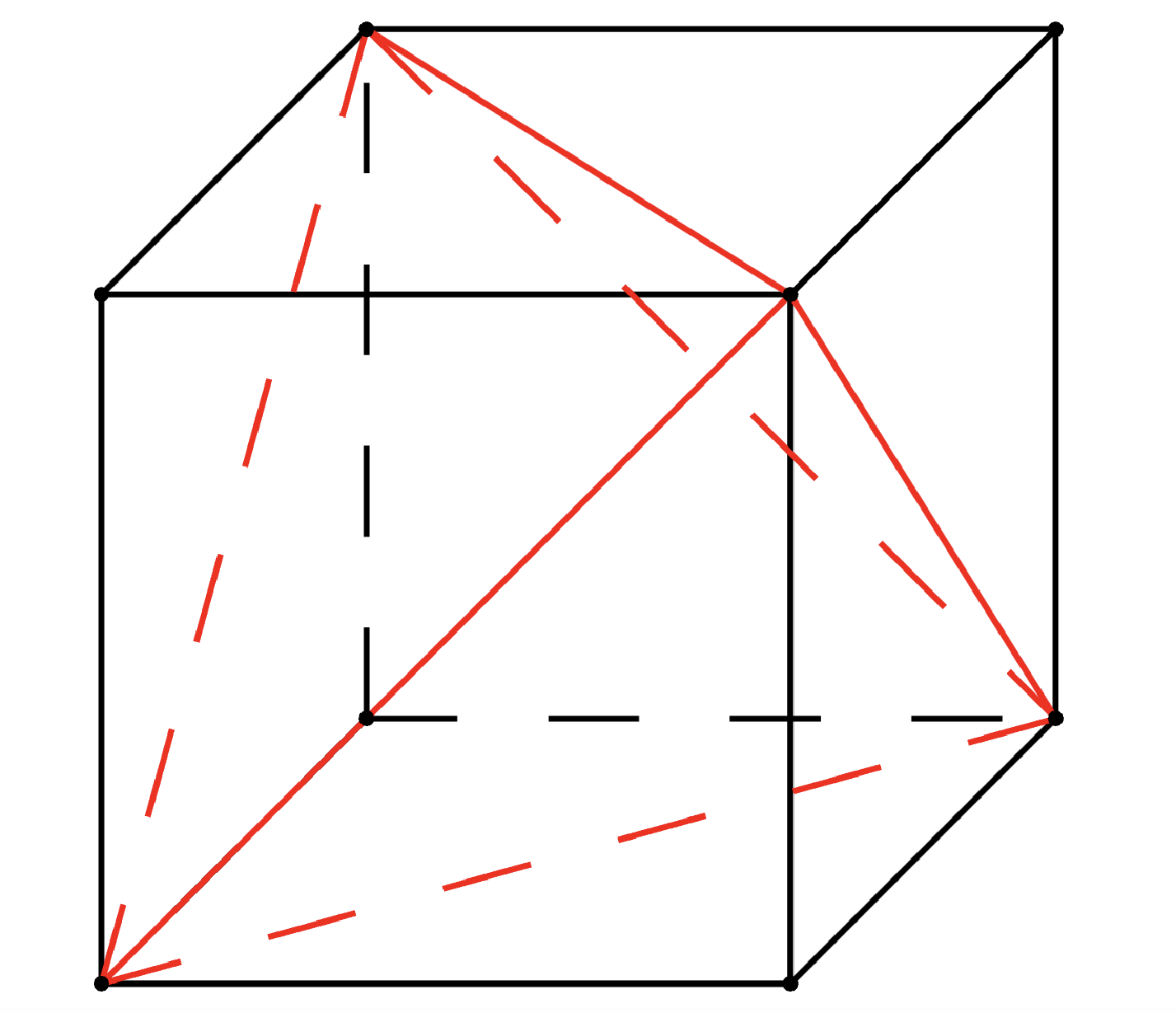}
			\centerline{(b)}
		\end{minipage}
\caption{Schematic diagram of the feasible domain: (a) system (I) and (b) system (II)}
	\end{figure}
\end{proof}\

\begin{remark}
Theorem \ref{StrongPLgC} means that PbLagCs, unlike PbBenCs, may cut off some region within the Lagrangian dual relaxation. It is thus possible to obtain a tighter relaxed feasible region than (\ref{strength}). Moreover, due to aggregation of scenarios, the number of scenarios looks ``smaller" in the aggregated problem (\ref{equa:parti}), which will lead to less time to generate PbLagCs than to produce the Lagrangian cuts.   
\end{remark}

\section{Experimental studies}
In this section, we implement experimental studies to test the proposed Algorithm \ref{a1}. 

\subsection{Implementation details}\label{sec5.1}
Three classes of two-stage SIP problems, including stochastic service location problem (sslp), a variant of the stochastic service location problem (sslpv), and stochastic multi-commodity flow problem (smcf) with different sizes are considered. The sslp problem (\cite{ntaimo2005million}) is a two-stage SIP with pure binary first-stage and mixed-binary second-stage variables. In this problem, the decision maker has to choose from $n_1$ sites to allocate servers with cost in the first stage. Then in the second stage, the availability of each client would be observed and every available client must be served at some site also with cost. The objective is to minimize the total cost. Note that integer variables are involved in the second stage of this problem, but to adapt to our specific context, we relax the integer restriction here, as done by Song and Luedtke (\cite{song2015adaptive}). The smcf problem (\cite{crainic2001bundle}) contains pure binary first-stage and continuous second-stage variables, in which the decision maker has to choose some edges with capacity constraint from the node-edge graph to transfer commodity flows. Then in the second stage, the demand of each commodity is available and must be transferred from its original node to the destination node by the chosen edges. Some information about these problems is presented in Table \ref{T1}. In the subsequent experiments, the instances for sslp and sslpv are generated according to the method proposed by (\cite{chen2022generating}) while the smcf instances are drawn upon from the work of Crainic et al. (\cite{crainic2001bundle}). In smcf, we consider the instances 'r04' and the stochastic demands are generated following the approach of (\cite{rahmaniani2018accelerating}).

\begin{table}
\caption{Profiles of three classes of problems.}
\label{T1}
\begin{center}
\begin{tabular}{llllll}
\hline
Instances& $\lvert S \rvert$& $n_1^\S$& $n_2$& $m_1$& $m_2$\\
\hline
sslp(40-50)$^\dag$& [50, 200]$^*$& 40& 2040& 1& 90\\
sslp(30-70)& [50, 200]& 30& 2130& 1& 100\\
sslp(20-100)& [50, 200]& 20& 2020& 1& 120\\
sslpv(40-50)& [50, 200]& 40& 2040& 1& 90\\
sslpv(30-70)& [50, 200]& 30& 2130& 1& 100\\
sslpv(20-100)& [50, 200]& 20& 2020& 1& 120\\
smcf(r04.1-r04.6)& [500,1000]& 60& 600& 1&660 \\
\hline
\end{tabular}
\begin{tablenotes}
    \footnotesize
        \item $^\dag$The first and second digit represent the number of locations and of customers, respectively; $^*$There are two cases of $|S|=50$ and $|S|=200$; $^\S$$n_1$, $n_2$, $m_1$ and $m_2$ share the same meanings with those given in (\ref{equa:1}).   
\end{tablenotes}
\end{center}
\end{table}

In the experiments, we follow the method put forth by (\cite{chen2022generating}) to execute the Benders dual decomposition, and the generated Lagrangian cuts are separated through the cutting-plane technique. When generating PbLagCs, i.e., running Algorithm \ref{a1}, the master problem possesses a unified structure for consistence within each scenario partition while 
the separation problem adopts (\ref{equa12}). Based on the setting in (\cite{van2018adaptive}) and (\cite{pay2020partition}), the refinement parameters $\delta$ and $\kappa_1$ take
$\delta=2/n^{2}$ with $n$ to denote the number of refinements and $\kappa_1=0.2$. We also conduct comparative experiments using the Benders dual decomposition algorithm and the classic Benders decomposition algorithm on the mentioned three classes of two-stage SIP problems. In addition, note that the refinement parameters have a large effect on experimental results, e.g., $\delta$ will affect the effectiveness of the adaptive algorithm greatly, we thus further provide testing experiments of Algorithm \ref{a1} on different refinement parameters. \

All the experiments are conducted on a Windows laptop with 8GB RAM and an Intel Core i5-7200U processor running at 2.5GHz with the
optimization solver Gurobi 9.5.0 and the Python (VS Code) compiler environment. The solving time limit for Algorithm \ref{a1} is set to be 1h. During the process of Benders dual decomposition, we stop generating the Lagrangian cuts when the gap closed by the last five iterations is less than $5\%$ of the total gap closed so far or the time limit is beyond. The following notations will be used to identify the details of our experimental results.

- \textbf{Partition-b$\&$d}: partition-based Benders dual decomposition (Algorithm \ref{a1});

- \textbf{B$\&$D}: Benders dual decomposition according to (\cite{chen2022generating});

- \textbf{Benders}: classic Benders decomposition to the linear relaxation of the problem;

- \textbf{T}: computational time (the unit is second) of an algorithm for the instance;

- \textbf{Ccut}: the number of PbLagCs added in the algorithm;

- \textbf{Fcut}: the number of Lagrangian cuts or Benders cuts added in the algorithm;

- \textbf{Refine}: refinement operation times;

- \textbf{Lower bound}: the final optimal value of the master problem;

- \textbf{$\lvert \mathcal{N} \rvert$}: the final partition size.

\subsection{Results and discussions}
Based on the given experimental details, we undertake the corresponding experiments on the instances of sslp, sslpv and smcf, with the results shown in Table \ref{T2}, Table \ref{T3}, Table \ref{T4}, respectively. In those tables, the results of lower bounder, the number of cuts added (\textbf{Ccut} or \textbf{Fcut}) and the computational time (\textbf{T}) are reported for \textbf{Partition-b$\&$d}, \textbf{B$\&$D} and \textbf{Benders}, where the optimal results are identified by bold for every instance. As can be seen, for most of the instances our proposed APbLagC algorithm can surpass the Benders dual decomposition algorithm both in the time required for enhancing the lower bound and in the count of cuts incorporated into the master problem. Moreover, the lower bound raised is universally higher through APbLagC algorithm in many instances but with the associated computational time significantly reduced. The main reason is that the size of their final scenario partition, i.e., $\lvert \mathcal{N} \rvert $, is notably smaller than that of the corresponding initial scenario configurations $\lvert S \rvert $. Of course, there are some instances for which our algorithm performs not so well and even worse than original bender dual decomposition, such as those bold results in the \textbf{B$\&$D} column of Table \ref{T3}. The possible reason may be that APbLagC algorithm cannot exhibit advantage for some instances with small size of scenarios. Additionally, it should be mentioned that some instances has run over the time limit 3600s for \textbf{B$\&$D}, the reason of which is that we need to implement the whole round of generating cuts. \

\begin{table}[htbp]
\centering
\rotatebox{90}{
\begin{threeparttable}
\begin{tabular}{lllllllllllll}\\
\hline
 Instances& $\lvert S \rvert$  & \multicolumn{5}{l}{Partition-b$\&$d} & \multicolumn{3}{l}{B$\&$D} & \multicolumn{3}{l}{Benders}
 \\ 
 &  &Lower bound       & Ccut      &  T  &refine & $\lvert \mathcal{N} \rvert$ &  Lower bound    & Fcut      & T      &  Lower bound     &   Fcut      &T   \\
 \hline
 sslp1(40-50)& 50 &   \textbf{-414}$^*$   &    78 &   \textbf{1310}   & 2 & 16& -414   &    404   &   2263   &    -594   &    705   &    7  \\
 sslp1(40-50)& 200 &   \textbf{-406}   &    156    &  \textbf{2164}   & 1 &  4 &  -441   &   1151    &   3600   &  -580    &   2590    &    16  \\
sslp1(20-100)& 50 &   \textbf{-884}   &    537   &  \textbf{1736}   & 3&50 & -884   &   910    &  3000    &    -1175   &    617   &   5   \\
sslp1(20-100)& 200 &    \textbf{-887}   &    90   &   \textbf{430}  & 1  & 2  &  -894   &   2000    &   3600   &   -1176   &    2359   &   13   \\
sslp1(30-70)& 50 &    \textbf{-599}   &   267   &   \textbf{2111}   & 2 &  29  & -599    &    551   &   2696   &  -812     &   565    &   10   \\
sslp1(30-70)& 200 &    \textbf{-603}   &   203    &   \textbf{984}   & 2& 70  &-627    &   1520   &   3600   &    -824   &    2690   &    15  \\

 sslp2(40-50)& 50 &    \textbf{-373}   &    91   &   \textbf{1700}   &  1 & 3  &  -392    &  479     &  3600    &    -566   &  546     &  10    \\
 sslp2(40-50)& 200 &   \textbf{-409}    &   55    &   \textbf{395}   & 1 &  12  & -430    &    1243   &  3600    &   -583    &   2443    &   22   \\

sslp2(30-70)& 50 &   \textbf{-560}    &  176     &   \textbf{3026}   & 1 &  9 & -560   &  787     &  3600    &    -804   &    734   &   9   \\
 sslp2(30-70)& 200  &   \textbf{-569}    &  130    &   3177   & 2 &  7  & -620    &   1408    &    \textbf{3052}  &   -821    &   2708    &   20.3   \\
 sslp2(20-100)& 50 &  -883     &   121    &   2800   & 2& 48  &-883     &  913     &   \textbf{2600}   &   -1185    &   503    &   7   \\
 sslp2(20-100)& 200 &     \textbf{-888}  &   241    &   \textbf{1009}   &  1  &  17 
 &    -888  &   2349    &   3600   &   -1175    &    2004   &    11  \\
\hline
\end{tabular}
\begin{tablenotes}
    \footnotesize
        \item $^*$The bold identification represents the optimal result among \textbf{Partition-b$\&$d}, \textbf{B$\&$D} and \textbf{Benders} for every instance, which also applies to \ref{T3} and \ref{T4}.   
\end{tablenotes}
\end{threeparttable}
}
\caption{Experimental results on sslp}
\label{T2}
\end{table}

\begin{table}[htbp]
\centering
\rotatebox{90}{
\begin{tabular}{llllllllllllll}\\
\hline
 Instances& $\lvert S \rvert$  & \multicolumn{5}{l}{Partition-b$\&$d} & \multicolumn{3}{l}{B$\&$D} & \multicolumn{3}{l}{Benders} \\ 
 &  &Lower bound       & Ccut      &  T  &Refine & $\lvert \mathcal{N} \rvert$ & Lower bound    & Fcut      & T      &  Lower bound     &   Fcut      &T   \\ 
 \hline
 sslpv1(40-50)& 50 &    -484   &    161   &   3600   & 1 &7 &\textbf{-471}   &     702  &   3600   &   -594   &    753   &    10  \\
 sslpv1(40-50)& 200 &   \textbf{ -495}  &     25   &  \textbf{219}  &1&3 &   -496   &  1165   &    3600  &   -580   &   2286    &  16    \\
sslpv1(20-100)& 50 &   -999    &   248    &   \textbf{910}   & 2 & 26&  \textbf{-995}    &    974   &    2700  &   -1175    &  549     &    8  \\
sslpv1(20-100)& 200 &   \textbf{-998}    &    136   &   \textbf{537} &  2& 20 &   -1004   &    1934   &  3600    &   -1176    &   2193    &   15   \\
sslpv1(30-70)& 50 &   -679   &  93   &   \textbf{927}   & 1 & 5 &  \textbf{-675}    &   671    &  3600   &  -812     &   678    &   8   \\
sslpv1(30-70)& 200 &    \textbf{-690}   &    401   &   \textbf{1569}   & 3 &  95&  -704  &  1356     &   3600   &   -824    &   2063    &   15   \\

 sslpv2(40-50)& 50 &    -461   &   109    &   \textbf{1370}   & 2&  16& \textbf{-460}    &  760     &  3600    &   -571    &    634   & 8  \\
 sslpv2(40-50)& 200 &  \textbf{-481}     &    207   &   \textbf{1337}  & 1& 28 &   -490    &      1168 &    3600  &   -583    &   2393    &   16   \\

sslpv2(30-70)& 50 &   \textbf{-654}    &    179   &   \textbf{1542}   & 2 &23  &  -658  &   687  &  3001 &    -804   &    637  &   8   \\
 sslpv2(30-70)& 200  &    \textbf{-678}   &  151    &   \textbf{1396} &1 & 10 &   -692    &     1367  &  3600    &   -821    &    2762   &    16  \\
 sslpv2(20-100)& 50 &   \textbf{1003}    &   201    &   \textbf{2294}  & 2& 12&   -1003   &    769   &  2770    &    -1185   &    584  &   6   \\
 sslpv2(20-100)& 200 &   \textbf{-998}    &   138    &   \textbf{1618}  & 2&7 &   -1016   &     1580  &  2200    & -1175     &   2058    &  20   \\
\hline
\end{tabular}
}
\caption{Experimental results on sslpv}
\label{T3}
\end{table}

\begin{table}[htbp]
\centering
\rotatebox{90}{
\begin{tabular}{llllllllllllll}\\
\hline
 Instances& $\lvert S \rvert$  & \multicolumn{5}{l}{Partition-b$\&$d} & \multicolumn{3}{l}{B$\&$D} & \multicolumn{3}{l}{Benders} \\ 
 &  &Lower bound       & Ccut      &  T &Refine& $\lvert \mathcal{N} \rvert$ &   Lower bound    & Fcut      & T      &  Lower bound     &   Fcut      &T   \\
 \hline
smcf(r04.1)   & 1000 &  \textbf{32730}     &    4004   &  \textbf{1950}    &  2 &156& 32183   &    8799   &    3300  &   29672    &   45407    &   190   \\
smcf(r04.2)  & 1000 &    \textbf{48864}   &    3990   &  \textbf{2889}   &  2& 391&  47702     &   5901    &  3600    &   37534    &   34893    &    159  \\
smcf(r04.3)   & 1000 &    \textbf{67375}   &   1953    &     \textbf{3793} & 1& 138&  62500    &    3882   &    3800  &    45115   &    26039   &    115  \\
smcf(r04.4)  & 1000 &   \textbf{34680}   &   6328    &   3600   & 1 & 368&  34466   & 5820      &  3600    & 32824  &    37847   &   173   \\
smcf(r04.5)  &1000& \textbf{53276} &    2795   &   \textbf{3536}   & 1 & 348&  50870  &   2983    &   4670    &    45641  &    39241 & 83      \\
smcf(r04.6)  & 1000 &   \textbf{77011}    &    1180  & \textbf{3600}    & 1 & 96 &   69271   &    2910   &  4925    &  59952 &  34152     &   167   \\
smcf(r04.1)  & 500 &   \textbf{32605}    &    920  &  \textbf{1000}   & 2&92 &   32172    &    600   &   1439   & 29626 &  21924   &    97\\
smcf(r04.2)  & 500 &    \textbf{50186}  &  2619    &   \textbf{2526}   &1 & 153&   48527   &    3911   &  3394    &  37424 &   15848   &   73  \\
smcf(r04.3)  & 500 &    \textbf{68499}  &   1230   &  \textbf{3338}  & 2& 289&  63197     &   1400    &   3600   & 45040  &    15040   &   69   \\
smcf(r04.4)  & 500 &   34381    &    3588  &    3581  & 1& 275 & \textbf{34434}     &   4978    &   \textbf{3444}   &  32626&  33924     &99\\
smcf(r04.5)  & 500 &    \textbf{54057}   &    1760 &   \textbf{3210}  &1 & 193 &   52748   &    1156   &     3619 &  45404 &    20228   &     160 \\
smcf(r04.6)  & 500 &   \textbf{74212}    &    1789  &3600& 2 & 199 & 702231    &   1960    &    3600  & 59515  &   17786    &  80    \\
\hline
 \end{tabular}
}
\caption{Experimental results on smcf}
\label{T4}
\end{table}

In order to observe more visually the evolving trends of lower bounds over time obtained by \textbf{Partition-b$\&$d} and \textbf{B$\&$D}, we display some representative instances for sslp, sslpv and smcf in Figure \ref{fig3}, Figure \ref{fig4} and Figure \ref{fig5}, respectively. It is obvious that our proposed adaptive algorithm can improve the lower bound rapidly. All of the exhibited curves rise rapidly towards the optimal values at the beginning period, and then slow down gradually after implementing refinement. This further demonstrates that APbLagCs algorithm is very efficient in solving two-stage SIP problems compared to \textbf{B$\&$D}. It also needs to be pointed out that for some instances, APbLagCs algorithms end with nearly the same time and lower bound as \textbf{B$\&$D}, which may result from a relatively large size of scenario partitions needed for the former during solving, and a great deal of time is consumed to get slight improvement of lower bound, such as the instance r04.2-1000 in Figure \ref{fig5} (b) and Table \ref{T4}. Naturally, there are some rooms to improve the efficiency of our proposed algorithm, such as considering more elaborate (heuristic) refinement policy and termination condition. Here, we will not give more discussions about it, but keep which as a point of future studies.

\begin{figure}[htbp]

	\begin{minipage}{0.47\linewidth}
		\vspace{3pt}
		\centerline{\includegraphics[width=\textwidth]{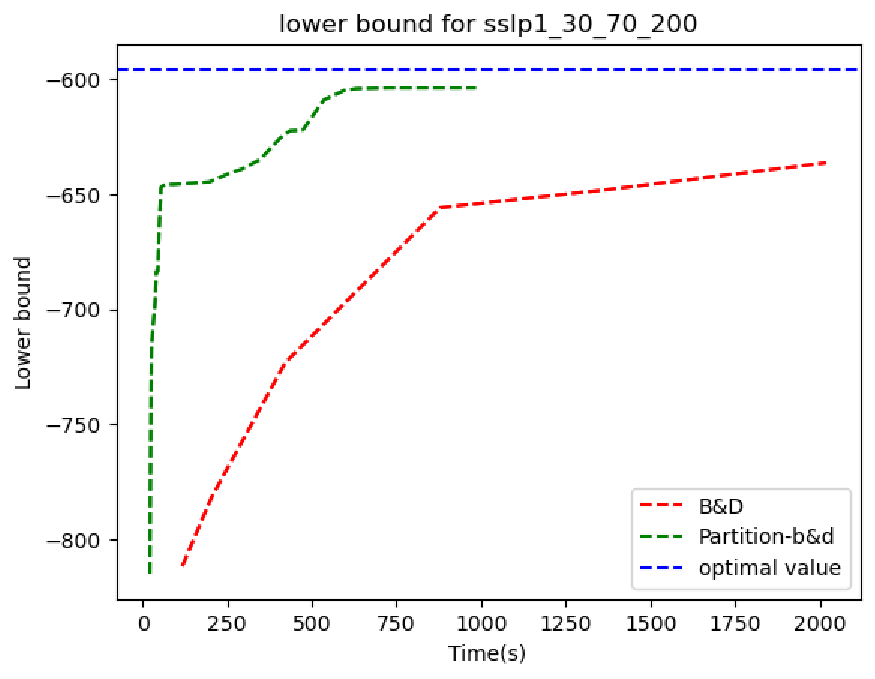}}
	 
		\centerline{(a)}
	\end{minipage}
	\begin{minipage}{0.47\linewidth}
		\vspace{3pt}
		\centerline{\includegraphics[width=\textwidth]{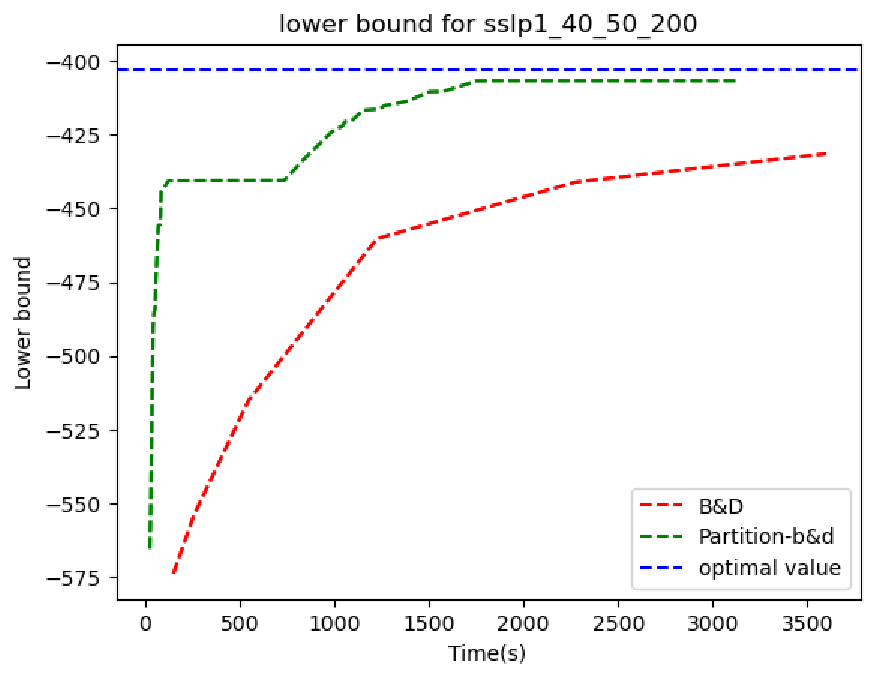}}
	 
		\centerline{(b)}
	\end{minipage}
 
	\caption{Visualization of sslp instances: (a) sslp1\_30\_70\_200; (b) sslp1\_40\_50\_200.}
	\label{fig3}
\end{figure}

\begin{figure}[htbp]

	\begin{minipage}{0.47\linewidth}
		\vspace{3pt}
		\centerline{\includegraphics[width=\textwidth]{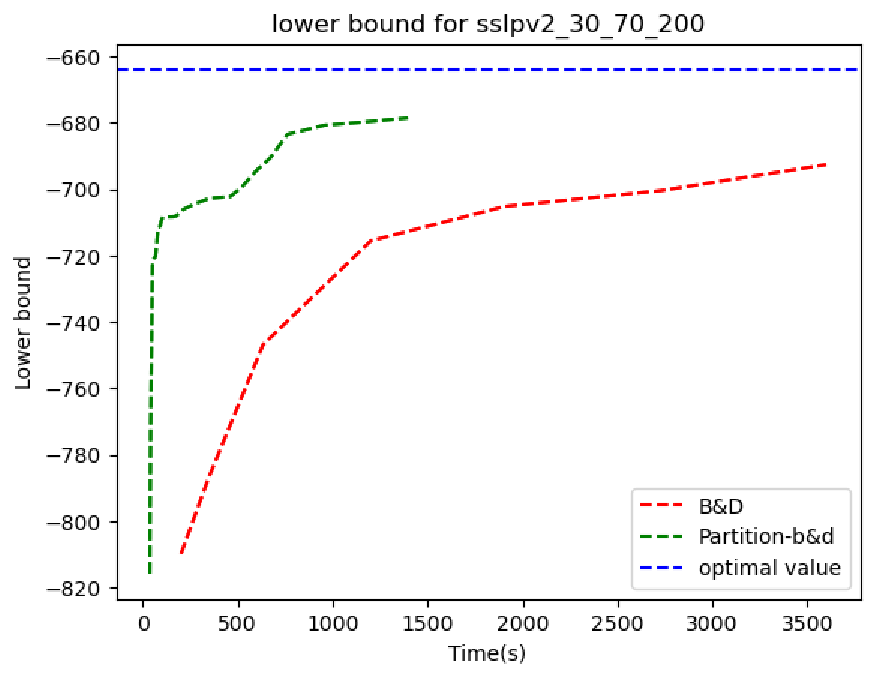}}
	 
		\centerline{(a)}
	\end{minipage}
	\begin{minipage}{0.47\linewidth}
		\vspace{3pt}
		\centerline{\includegraphics[width=\textwidth]{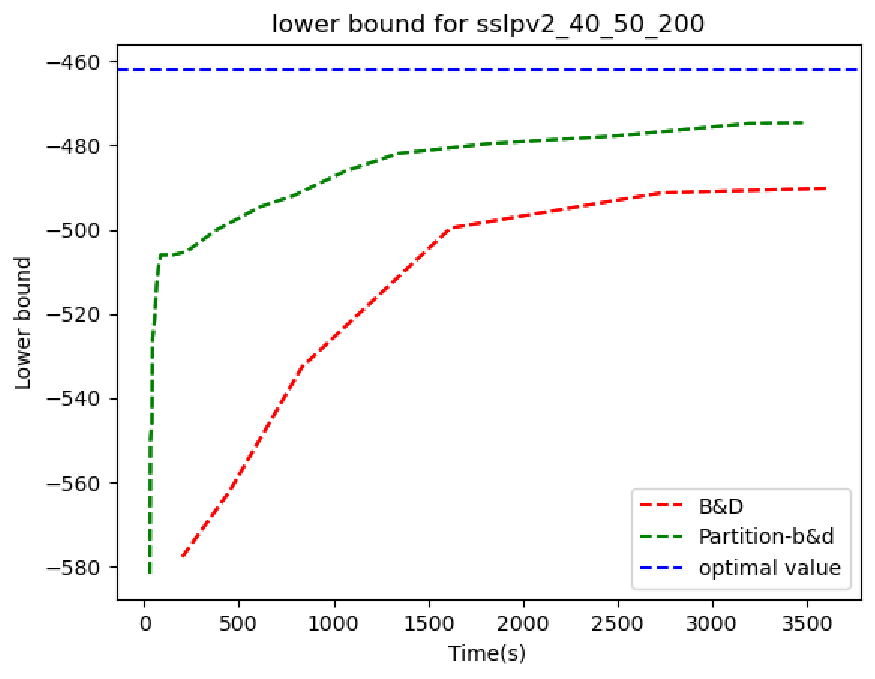}}
	 
		\centerline{(b)}
	\end{minipage}
 
	\caption{Visualization of sslpv instances: (a) sslpv2\_30\_70\_200; (b) sslpv2\_40\_50\_200.}
	\label{fig4}
\end{figure}

\begin{figure}[htbp]
	
	\begin{minipage}{0.47\linewidth}
		\vspace{3pt}
		\centerline{\includegraphics[width=\textwidth]{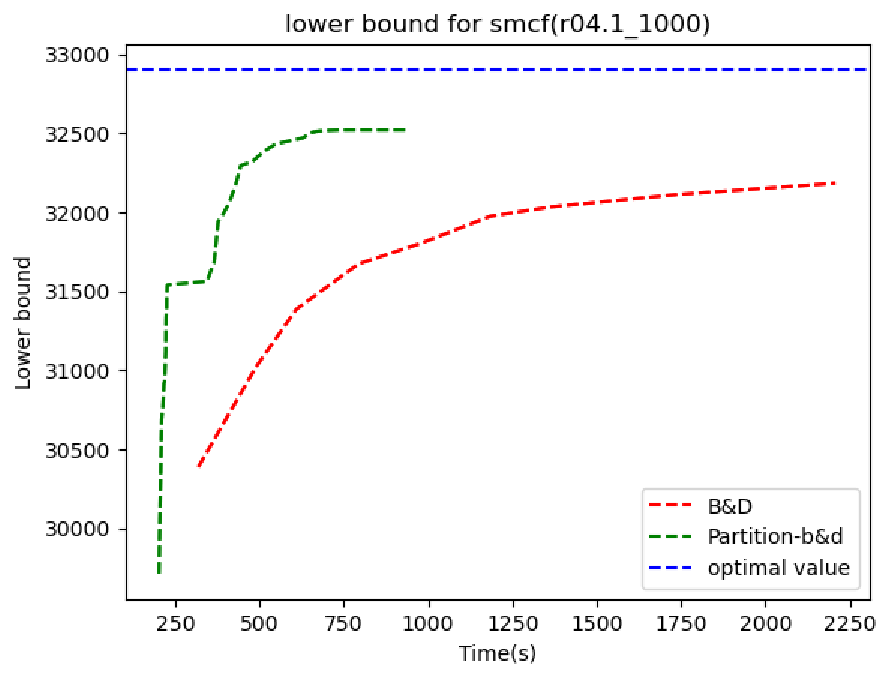}}
		\centerline{(a)}
	\end{minipage}
	\begin{minipage}{0.47\linewidth}
		\vspace{3pt}
		\centerline{\includegraphics[width=\textwidth]{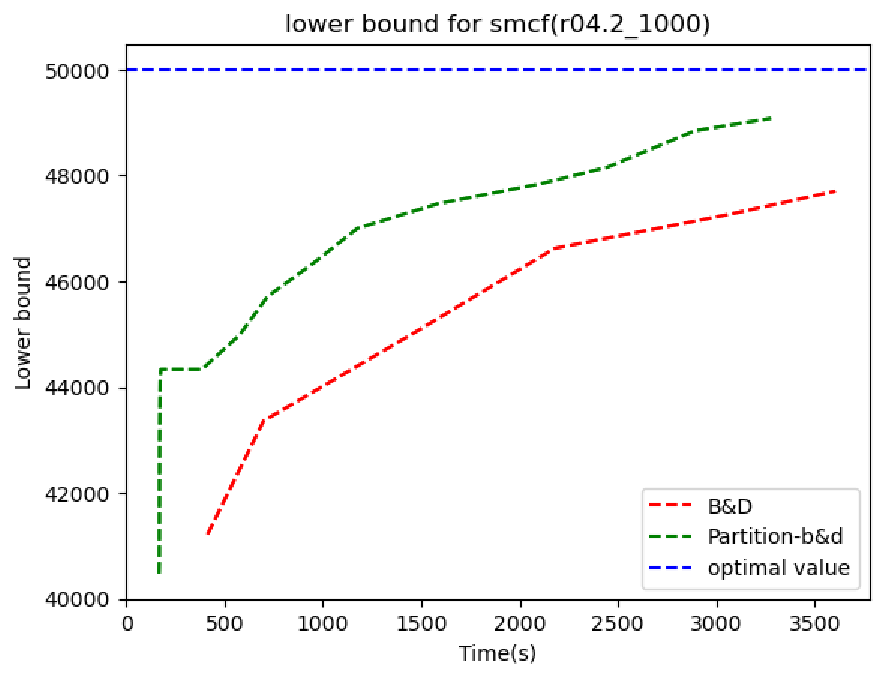}}
	 
		\centerline{(b)}
	\end{minipage}

	\caption{Visualization of smcf instances: (a) r04.1\_1000; (b) r04.2\_1000.}
	\label{fig5}
\end{figure}\

\begin{table}[htbp]
\caption{Experimental results of APbLagCs algorithm with different refinement parameter $\delta$}
\label{T6}
\centering
\begin{tabular}{llllllll}
\hline
 Instances& $\lvert S \rvert$  & \multicolumn{3}{l}{$\delta = 2/n^2$} & \multicolumn{3}{l}{$\delta = 1/n^2$}\\ 
 &  & LB$^*$       &   $\lvert \mathcal{N} \rvert$   &  T   &   LB   & $\lvert \mathcal{N} \rvert$     & T        \\ \hline
sslp1(40-50)& 50 &   -414   &    16 &   1310   & -416  &    49   &  3600   \\
 sslp1(40-50)& 200 &   -406   &    4    &  2164&    -429  &   86    &  3192 \\
sslp1(20-100)& 50 &   -884   &    50   &  1736  &  -884    &50&1736\\
sslp1(20-100)& 200 &    -887   &    2  &   430   &   -887  &2&   430   \\
sslp1(30-70)& 50 &    -599   &   29   &   2111   &-596&49&  2666       \\
sslp1(30-70)& 200 &    -603   &   70   &   984   &-629& 198  &  2963   \\
 sslp2(40-50)& 50 &    -373   &    3&   1700&   -373   &3&1700\\
 sslp2(40-50)& 200 &   -409    & 12    &   395   &   -409    &    26  &  430  \\
sslp2(30-70)& 50 &   -560    &  9&   3026   &   -560   &   46   &3364\\
 sslp2(30-70)& 200  &   -569    &  7    &   3177   &   -605   & 105  &   3600  \\
 sslp2(20-100)& 50 &  -883     &   48    &   2800   &  -883   &20&  2342  \\
 sslp2(20-100)& 200 &     -888  &   17  &   1009  &   -887   &    123   &   3380  \\
 
  sslpv1(40-50)& 50 &    -484   &    7&   3461   &  -487    &45& 3557  \\
 sslpv1(40-50)& 200 &    -495  &     3   &  219  &     -495 & 3    & 219     \\
sslpv1(20-100)& 50 &   -999    &   26    &   910  & -999     &26&910 \\   
sslpv1(20-100)& 200 &   -998    &    20   &   537&   -998   &   19    &  693 \\
sslpv1(30-70)& 50 &   -679   &  5   &   927&   -679   &    5  &  927\\
sslpv1(30-70)& 200 &    -690   &    95&   1569   & -690   & 6 &   1023  \\

 sslpv2(40-50)& 50 &    -461   &   16   &   1370&  -461    & 16 & 1370  \\
 sslpv2(40-50)& 200 &  -481     &    28&   1337  &  -483 &      46 &  3145    \\

sslpv2(30-70)& 50 &   -654    &    23   &   1542   & -654     &  28   &  1650  \\
 sslpv2(30-70)& 200  &    -678   &  10    &   1396  &  -678    &  17  &  2358  \\
 sslpv2(20-100)& 50 &   -1003    &  12   &   2294   &   -1003   &   4    &338\\
 sslpv2(20-100)& 200 &   -998    &   7   &   1618 &   -997  &  6    & 1133  \\
smcf(r04.1)   & 1000 &  32730     &    156   &  1950    &   32730    & 156      &  1950    \\
smcf(r04.2)  & 1000 &    48864  &    391   &  2889    &   48670    &   393    & 3600 \\
smcf(r04.3)   & 1000 &    67375  &    138   &     3793 &   67375   &   138    & 3793 \\
smcf(r04.4)  & 1000 &   34680   &   368   &   3600   &  34583     &    377   &  3600  \\
smcf(r04.5)  &1000& 53276 &   348   &   3536   &   53276   &   348   &    3536    \\
smcf(r04.6)  & 1000 &   77011    &    96&  3600    &  77011  &  96    &  3600    \\
smcf(r04.1)  & 500 &   32605    &    92  &  1000    &      32605 &92       &1000\\
smcf(r04.2)  & 500 &    50186   &  153&   2526   &   50186    &   153    &    2526  \\
smcf(r04.3)  & 500 &    68499   &   289   &  3338    &   68499    &   289    &    3338    \\
smcf(r04.4)  & 500 &   34381    &    275  &    3581  &    34386   &   278   &3021\\
smcf(r04.5)  & 500 &     54057   &    193 &   3210   &   54430    &    197   &   3477    \\
smcf(r04.6)  & 500 &   74212    &    199  &3600&      74128 &    201  &    3684   \\
\hline
 \end{tabular}
 \begin{tablenotes}
    \footnotesize
        \item $^*$LB means lower bound.   
\end{tablenotes}
 \end{table}
As mentioned in Subsection \ref{sec5.1}, the refinement parameter $\delta$ has a large effect on the experimental results, which is set to be $2/n^{2}$ in the above experiments. In the following, we will test its impact by setting it to be different values, including $\delta = 2/n^2$ and $\delta = 1/n^2$. To compare more intuitively, we treat the result at $\delta = 2/n^2$ as the baseline, that is when implementing the partition-based algorithm for $\delta = 1/n^2$, we attempt to make the lower bound attain the former's, and then to observe their running time. Table \ref{T6} reports the corresponding results. As can be seen, for instances of sslp and sslpv $\delta$ has a tremendous impact to the algorithmic effectiveness. Basically, $\delta = 1/n^2$ may result in larger scenario partition size and therefore consume much more time.  Nevertheless, it is not always the case that the larger $\delta$ performs better. The exceptions happen to instances of sslp2\_20\_100\_50, sslpv1\_30\_70\_200, sslpv2\_20\_100\_50 and sslpv2\_20\_100\_200, where smaller $\delta$ results in a smaller final scenario partition size and consumes less running time. This is because the larger size of scenario partition in the initial rounds reduces the number of refinements in these instances. For instances of smcf, the impact of $\delta$ is relatively small. In addition, we also test the situation where $\delta=0.5/n^2$, in which the result is similar to that at $\delta = 1/n^2$, and we thus ignore it here. More research may be needed on analyzing how to set hyper-parameters and developing more effective heuristic refinement methods to keep the scenario partition size as small as possible.

\section{Conclusion}
In this paper, we develop the APbLagCs algorithm based on Benders dual decomposition for solving two-stage SIPs with continuous recourse. The generated partition-based Lagrangian cuts are proved not necessarily dominated by any Lagrangian cut. We conduct extensive experiments using the algorithm on various test instances commonly employed in the literature. Our experiments demonstrate that the proposed algorithm outperforms Benders dual decomposition in terms of computational time, the number of Lagrangian cuts added, and the ability to enhance the lower bound. However, similar to previous adaptive algorithms for two-stage SIP, the proposed approach is applicable only to cases with continuous recourse. Therefore, a potential future research direction is to extend the adaptive framework to encompass more general scenarios. Furthermore, we plan to investigate how the structural characteristics of the SIP influence the lower bound of Lagrangian relaxation during scenario partition implementation and  more effective heuristic refinement technology.
\ACKNOWLEDGMENT{%
This work was funded by the National Nature Science Foundation of China under Grant No. 12320101001 and 12071428.
}

%
%
%


\bibliographystyle{informs2014trsc} 
\bibliography{references} 

\begin{thebibliography}{23}
\providecommand{\natexlab}[1]{#1}
\providecommand{\url}[1]{\texttt{#1}}
\providecommand{\urlprefix}{URL }

\bibitem[{Ahmed \protect\BIBand{} Shapiro(2002)}]{ahmed2002sample}
Ahmed S, Shapiro A, 2002 \emph{The sample average approximation method for stochastic programs with integer recourse. optimization online}.

\bibitem[{Birge \protect\BIBand{} Louveaux(2011)}]{birge2011introduction}
Birge JR, Louveaux F, 2011 \emph{Introduction to stochastic programming} (Springer Science \& Business Media).

\bibitem[{BnnoBRs(1962)}]{bnnobrs1962partitioning}
BnnoBRs J, 1962 \emph{Partitioning procedures for solving mixed-variables programming problems}. \emph{Numer. Math} 4(1):238--252.

\bibitem[{Bodur et~al.(2017)Bodur, Dash, G{\"u}nl{\"u}k, \protect\BIBand{} Luedtke}]{bodur2017strengthened}
Bodur M, Dash S, G{\"u}nl{\"u}k O, Luedtke J, 2017 \emph{Strengthened benders cuts for stochastic integer programs with continuous recourse}. \emph{INFORMS Journal on Computing} 29(1):77--91.

\bibitem[{Car{\o}e \protect\BIBand{} Schultz(1999)}]{caroe1999dual}
Car{\o}e CC, Schultz R, 1999 \emph{Dual decomposition in stochastic integer programming}. \emph{Operations Research Letters} 24(1-2):37--45.

\bibitem[{Chen \protect\BIBand{} Luedtke(2022)}]{chen2022generating}
Chen R, Luedtke J, 2022 \emph{On generating lagrangian cuts for two-stage stochastic integer programs}. \emph{INFORMS Journal on Computing} 34(4):2332--2349.

\bibitem[{Crainic, Frangioni, \protect\BIBand{} Gendron(2001)}]{crainic2001bundle}
Crainic TG, Frangioni A, Gendron B, 2001 \emph{Bundle-based relaxation methods for multicommodity capacitated fixed charge network design}. \emph{Discrete Applied Mathematics} 112(1-3):73--99.

\bibitem[{Duran-Mateluna, Al{\`e}s, \protect\BIBand{} Elloumi(2023)}]{duran2023efficient}
Duran-Mateluna C, Al{\`e}s Z, Elloumi S, 2023 \emph{An efficient benders decomposition for the p-median problem}. \emph{European Journal of Operational Research} 308(1):84--96.

\bibitem[{Florio, Hartl, \protect\BIBand{} Minner(2020)}]{florio2020new}
Florio AM, Hartl RF, Minner S, 2020 \emph{New exact algorithm for the vehicle routing problem with stochastic demands}. \emph{Transportation Science} 54(4):1073--1090.

\bibitem[{Kim, Pasupathy, \protect\BIBand{} Henderson(2015)}]{kim2015guide}
Kim S, Pasupathy R, Henderson SG, 2015 \emph{A guide to sample average approximation}. \emph{Handbook of simulation optimization} 207--243.

\bibitem[{Kleywegt, Shapiro, \protect\BIBand{} Homem-de Mello(2002)}]{kleywegt2002sample}
Kleywegt AJ, Shapiro A, Homem-de Mello T, 2002 \emph{The sample average approximation method for stochastic discrete optimization}. \emph{SIAM Journal on optimization} 12(2):479--502.

\bibitem[{Leutwiler \protect\BIBand{} Corman(2022)}]{leutwiler2022logic}
Leutwiler F, Corman F, 2022 \emph{A logic-based benders decomposition for microscopic railway timetable planning}. \emph{European Journal of Operational Research} 303(2):525--540.

\bibitem[{Ntaimo \protect\BIBand{} Sen(2005)}]{ntaimo2005million}
Ntaimo L, Sen S, 2005 \emph{The million-variable “march” for stochastic combinatorial optimization}. \emph{Journal of Global Optimization} 32:385--400.

\bibitem[{Oliveira, de~S{\'a}, \protect\BIBand{} de~Souza(2022)}]{oliveira2022benders}
Oliveira FA, de~S{\'a} EM, de~Souza SR, 2022 \emph{Benders decomposition applied to profit maximizing hub location problem with incomplete hub network}. \emph{Computers \& Operations Research} 142:105715.

\bibitem[{Pavi{\'c} et~al.(2014)Pavi{\'c}, Wu, Klari{\'c} et~al.}]{pavic2014convex}
Pavi{\'c} Z, Wu S, Klari{\'c} {\v{S}}, et~al., 2014 \emph{Convex combination inequalities of the line and plane}. \emph{Abstract and Applied Analysis}, volume 2014 (Hindawi).

\bibitem[{Pay \protect\BIBand{} Song(2020)}]{pay2020partition}
Pay BS, Song Y, 2020 \emph{Partition-based decomposition algorithms for two-stage stochastic integer programs with continuous recourse}. \emph{Annals of Operations Research} 284(2):583--604.

\bibitem[{Rahmaniani et~al.(2020)Rahmaniani, Ahmed, Crainic, Gendreau, \protect\BIBand{} Rei}]{rahmaniani2020benders}
Rahmaniani R, Ahmed S, Crainic TG, Gendreau M, Rei W, 2020 \emph{The benders dual decomposition method}. \emph{Operations Research} 68(3):878--895.

\bibitem[{Rahmaniani et~al.(2017)Rahmaniani, Crainic, Gendreau, \protect\BIBand{} Rei}]{rahmaniani2017benders}
Rahmaniani R, Crainic TG, Gendreau M, Rei W, 2017 \emph{The benders decomposition algorithm: A literature review}. \emph{European Journal of Operational Research} 259(3):801--817.

\bibitem[{Rahmaniani et~al.(2018)Rahmaniani, Crainic, Gendreau, \protect\BIBand{} Rei}]{rahmaniani2018accelerating}
Rahmaniani R, Crainic TG, Gendreau M, Rei W, 2018 \emph{Accelerating the benders decomposition method: Application to stochastic network design problems}. \emph{SIAM Journal on Optimization} 28(1):875--903.

\bibitem[{Ram{\'\i}rez-Pico, Ljubi{\'c}, \protect\BIBand{} Moreno(2023)}]{ramirez2023benders}
Ram{\'\i}rez-Pico C, Ljubi{\'c} I, Moreno E, 2023 \emph{Benders adaptive-cuts method for two-stage stochastic programs}. \emph{Transportation Science} .

\bibitem[{Song \protect\BIBand{} Luedtke(2015)}]{song2015adaptive}
Song Y, Luedtke J, 2015 \emph{An adaptive partition-based approach for solving two-stage stochastic programs with fixed recourse}. \emph{SIAM Journal on Optimization} 25(3):1344--1367.

\bibitem[{van Ackooij, de~Oliveira, \protect\BIBand{} Song(2018)}]{van2018adaptive}
van Ackooij W, de~Oliveira W, Song Y, 2018 \emph{Adaptive partition-based level decomposition methods for solving two-stage stochastic programs with fixed recourse}. \emph{INFORMS Journal on Computing} 30(1):57--70.

\bibitem[{Zhang \protect\BIBand{} Kucukyavuz(2014)}]{zhang2014finitely}
Zhang M, Kucukyavuz S, 2014 \emph{Finitely convergent decomposition algorithms for two-stage stochastic pure integer programs}. \emph{SIAM Journal on Optimization} 24(4):1933--1951.

\end{thebibliography}


\end{document}